\title{Linear foliations of complex spheres I. Chains}
\author{Laurent Dufloux}

\documentclass{article}

   \usepackage{amsmath}
    \usepackage{amsthm}
	\usepackage{amssymb}
\newcommand{\GB}{\mathbf{G}}
\newcommand{\DB}{\mathbf{D}}
\newcommand{\R}{\mathbf{R}}
\newcommand{\C}{\mathbf{C}}
\newcommand{\K}{\mathbf{K}}

\newcommand{\PB}{\mathbf{P}}
\newcommand{\inprod}[2]{\mathbf{(} #1\ |\ #2 \mathbf{)}}
\newcommand{\Span}{\mathrm{Span}}
\newcommand{\Dist}[2]{\frac{\|#1 \vee #2\|}{\|#1\| \cdot \| #2 \|}}
\newcommand{\Det}{\mathrm{det}}
\newcommand{\Proj}{\mathrm{Proj}}
\newcommand{\Leb}{\mathrm{Leb}}
\newcommand{\BB}{\mathbf{B}}
\newcommand{\SB}{\mathbf{S}}
\newtheorem{remark}{Remark}
\newtheorem{lemma}{Lemma}
\newtheorem{theorem}{Theorem}
\newtheorem{proposition}{Proposition}
\newtheorem{corollary}{Corollary}
\newtheorem{definition}{Definition}

\begin{document}
\maketitle
\abstract{We provide coordinate-free versions of the classical projection Theorem of Marstrand--Kaufman--Mattila. This allows us to generalize
this Theorem to the complex setting; in restriction to complex spheres, we obtain further projection Theorems along so-called \emph{complex chains}.}

\section{Introduction}

\subsection{Motivation}
\subsubsection{Background}
Since Marstrand's seminal paper \cite{Marstrand1954}, several authors have sought to improve and generalize the so-called
Marstrand projection Theorem.

Let us recall the basic Euclidean setup in any dimension, as in \cite{Mattila1995} (Marstrand's paper dealt only with the plane).
Let $n \geq 2$ and $1 \leq k \leq n-1$. Fix a Borel subset $A \subset \R^n$ of Hausdorff dimension $s$. 
Pick a vector subspace of dimension $k$ at random (with respect to the Lebesgue measure on the space of 
$k$--dimensional vector subspaces of $\R^n$). Project $A$ along the vector subspace, \emph{e.g.} 
pushing it down the quotient mapping. The projected set (sitting in a space isometric to $\R^{n-k}$) has, 
almost surely, Hausdorff dimension $\inf\{s, n-k\}$.

A Marstrand-type result is a Theorem of this kind: an almost sure equality for the dimension of a given Borel set projected 
through a random projection.

It could also be said that a Marstrand-type result deals with the almost sure dimension of a Borel set, transverse to
a random foliation (according to a fixed measure on some space of foliations). This is our point of view in this paper. For 
a definition, see \ref{ss.transverse-dimension}.

There are (at least) two natural ways to generalize this result: 
\begin{itemize}
	\item To look at a restricted family of foliations, \emph{e.g.} to consider, in $\R^3$, vector lines spanned not 
	by any vector but by a one-parameter family of vectors, as in \cite{Orponen2017}.
	\item To look at foliations defined in the same fashion in non-Euclidean spaces, \emph{e.g.} in Heisenberg group. 
	See, for example, \cite{Balogh2012}.
\end{itemize}

\begin{table}
	\caption{Notations}
	\vspace{1em}
	\begin{tabular}{|c|c|c|}\hline 
		Symbol & Page & Definition \\  \hline 
		$\GB(E)$ & \pageref{def.GB} & Grassmann algebra of $E$ \\
		$\GB^k (E)$ & \pageref{def.GBk} & $k$-vectors of $E$\\
		$\vee$ & \pageref{def.vee} & Progressive (exterior) product \\
		$\DB^k(E)$ & \pageref{def.DBk} & Decomposable $k$-vectors of $E$ \\
		$\Span(u_1 \vee \cdots \vee u_k)$ & \pageref{def.span} & $\K u_1 \oplus \cdots \oplus \K u_k$\\
		$\wedge$ & \pageref{def.wedge} & Regressive product \\
		$\GB^k (\Phi)$ & \pageref{def.GBkPhi} & Canonical extension of $\Phi$ to $\GB^k(E)$ \\
		$\PB_\K^n$ & \pageref{def.projective-space} & Projective space of dimension $n$ over $\K$ \\
		$\PB(E)$ & \pageref{def.PB} & Projective space associated to $E$\\
		$\PB \DB^k(\K^{n+1})$ & \pageref{def.PBDBK} & Image of $\DB^k (\K^{n+1})$ in the projective space $\PB \GB^k (\K^{n+1})$ \\
		$d$ & \pageref{eq.angular-metric} & Angular metric on a projective space \\
		$\perp$ & \pageref{def.perp} & Orthogonality w.r.t. inner product \\
		$\Proj_U$ & \pageref{def.projU} & Generalized radial projection at $U$\\
		$\tau$ & \pageref{def.tau} & See formula \eqref{def.tau} \\
		$\phi_U$ & \pageref{def.phiU} & Lipschitz modulus; see formula \eqref{def.phiU} \\
		$\Leb$ & \pageref{def.Leb} & Lebesgue measure \\
		$I_\sigma (\mu)$ & \pageref{def.energy} & $\sigma$-energy of $\mu$ \\
		$\SB,\BB$ & \pageref{def.S-B} & Sphere and open ball in $\PB_\K^n$ \\
		$\mathcal L_\K^k$ & \pageref{def.Lk} & Space of small $(k-1)$--spheres if $\K=\R$, resp. $k$--chains if $\K=\C$ \\ \hline 
	\end{tabular}
\end{table}

\subsubsection{Description the of results}
In this paper, we look at a quite obvious generalization: namely, 
we generalize Marstrand projection Theorem to the \emph{complex} setting. The word \emph{complex sphere}
in the title of this paper refers to the Euclidean spheres of odd dimension; these spheres sit naturally in
complex projective spaces, and the complex Marstrand projection Theorem we will obtain can be restricted to complex spheres 
to yield interesting projection Theorems with respect to some special families of so-called ``small spheres''.

The most notable feature of our approach is that everything happens in the projective space; this allows us to do 
coordinate-free computations, using extensively the Grassmann algebra. This adds some conceptual and notational difficulty.

The first half of our results (dealing with linear foliations of projective spaces) could be obtained using coordinates 
and standard computations as in \cite{Mattila1995} (where the real setting is handled). 

The other half of our results would be very awkward to formulate without the language of Grassmann algebra, especially in the complex 
setting, which is of interest to us. 
Informally, a reason for this is the fact that there is no good analogue, for Heisenberg groups, 
of projective spaces associated to vector spaces. The projective space $\PB_\R^n$ can be defined at the space of 
``infinite circles of $\R^{n+1}$ passing through the origin''. This definition would also make sense in Heisenberg group,
replacing ``infinite circles'' with ``infinite chains'' (see \cite{Goldman1999} for the definition 
of finite and infinite chains in Heisenberg group, or \ref{ss.general-setup} for the definition 
of chains we will be using); the point is that there is only one infinite chain passing 
through the origin. This is why, in this paper, we have to consider ``finite chains''; this also explains why it is 
much more efficient to work without coordinates.

Throughout this paper, we will deal with the real and complex settings at the same time. In the real setting, 
none of the results we obtain is new: they are all essentially equivalent to the basic Theorem 
of Marstrand--Kaufman--Mattila. We state them nonetheless because they serve to provide some geometric 
intuition to the reader, and to convince them that we are indeed generalizing the 
classical, real Euclidean, Marstrand Theorem -- this may not be obvious at first. 

\subsubsection{Plan of the paper}
In \ref{ss.transverse-dimension} we give a precise meaning to the notion of transverse dimension with respect to a foliation. 
In \ref{sss.grassmann-algebra} to \ref{sss.grassmann-extensions} we introduce the needed algebraic device: the Hermitian Grassmann (bi)algebra  
associated to a Hermitian space. In \ref{ss.basic-properties} we state and prove useful properties of the 
inner product in the Hermitian Grassmann algebra. The distance formula in \ref{ss.first-distance-formula} 
is a first hint of the geometric significance of the Hermitian structure on the Grassmann algebra.

Section \ref{s.linear-foliations} is the part of the paper that deals with linear foliations of projective
spaces and this is where we apply the algebraic tools described previously. In \ref{ss.radial-foliations}
we define the generalized radial projections we will use to parametrize our linear foliations. We then endow, in \ref{ss.codomain-metric}, the codomain 
of these radial projections with a canonical metric. The analysis of transversality of generalized radial projections is made easy by the 
product formula contained in \ref{ss.product-formula}. The result of \ref{ss.generalized-distance} is not needed in this paper; it is stated 
because it answers a question that arises naturally in this context. In \ref{ss.first-transversality} we prove transversality of the basic 
linear foliations of projective spaces, and this is applied to obtain a coordinate-free version of Marstrand's projection Theorem. 
We improve on this in \ref{ss.transversality-pointed} by looking at a lower-dimensional family of foliations; in the real setting the 
result we then obtain is equivalent to the classical Theorem of Marstrand--Kaufman--Mattila.

In Section \ref{s.spheres-foliations} we look at the results of the previous Sections in restriction to spheres.
The special case of $1$--chains has to
be dealt with separately in \ref{ss.foliations-one-chains}. In \ref{ss.concrete-look} we describe our foliations in coordinates
in order to help the reader get an idea of the geometry behind the algebra; a description of the geometry of 
chains is out of the scope of this paper and we refer to \cite{Goldman1999} for details, and suggestive computer-generated pictures. 

\subsection{Transverse dimension}\label{ss.transverse-dimension}
We give a precise meaning to the notion of \emph{dimension of a Borel set transverse to 
a given foliation}.

Let $X$ be a locally compact metric space. A \emph{foliation} on $X$ is a partition
$\xi$, all atoms of which are closed subsets of $X$. We denote by $\xi(x)$ the atom
of $\xi$ a given point $x \in X$ belongs to. The quotient space $X/\xi$ has elements
the atoms of $\xi$ and is endowed with the final topology for the projection
mapping $X \to X/\xi$. In general the metric of $X$ does not pass to the quotient,
because two distinct atoms of $\xi$ may be at zero distance from one another.
On the other hand, for any compact subset $K$ of $X$, the trace $\xi|K$ of $\xi$
on $K$ has compact atoms; the metric of $X$, restricted to $K$, passes
to the quotient, and the projection mapping $K \to K/\xi$ is, by definition,
Lipschitz.

\begin{definition} \label{def.transverse-dimension}
	In this situation, the \emph{transverse dimension of $K$ with respect
	to $\xi$} is the Hausdorff dimension of the quotient metric space $K/\xi$.

	If $A$ is a Borel subset of $X$, the \emph{transverse dimension of $A$
	with respect to $\xi$} is the supremum of the transverse dimensions
	of all compact subsets $K \subset A$, with respect to $\xi$.
\end{definition}
The transverse dimension of $A$ with respect to $\xi$ is at most equal to the Hausdorff dimension of $A$, $\dim A$;
indeed, $\dim A = \sup_K \dim K$ where $K$ goes through the family of compact subsets of $A$; and for any compact $K$,
the quotient mapping $K \to K/\xi$ is Lipschitz and cannot increase Hausdorff dimension.

This definition of transverse dimension highlights the fact that in general,
we do not have to be too concerned with the choice of the Lipschitz
mapping we use to parametrize the foliation. In the most classical situa-
tion, $X$ is the Euclidean plane and $\xi$ is the foliation of $X$ by affine lines
of some given angle $\theta$. The orthogonal projection onto the vector line
of angle $\theta + \pi/2$ is a suitable Lipschitz mapping, and so is the quotient
mapping with respect to the vector line of angle $\theta$.

In more complex situations, the adequate projection may have a less
elementary description, and it may also not be Lipschitz on the whole
space $X$, but only on every compact subspace of $X$. Our emphasis in this
paper will be on the geometry of foliations. We will introduce suitable
Lipschitz projections to work with but in our perspective the foliations
come first.

Let us provide a simple example to explain why it is useful to think in terms
of foliations rather than projections. Let $X$ be the Euclidean plane minus
the origin, and let $\xi$ be the foliation of $X$ by vector lines with the origin removed.
Now let $X'$ be the Euclidean plane minus two points $x$ and $y$, and
let $\xi'$ be the foliation of $X'$ by circles passing through $x$ and $y$ (with $x$
and $y$ removed from the circle). From the point of view of transverse dimension,
there is no difference between $(X, \xi)$ and $(X', \xi')$.
We may identify $X'$ with $X$ (we have to remove one more point from $X$)
via a M\"{o}bius transformation $f$; now $f$ maps $\xi'$ onto $\xi$ and it is locally
biLipschitz, so any local dimension property is preserved by $f$.

The fact that $\xi$ can be parametrized by the radial mapping at $0$ is
irrelevant and there is no need to find a corresponding projection mapping
for $\xi$.

Of course the transverse dimension of a given subset $A$ with
respect to $\xi$ depends of how $A$ is sitting with respect to $\xi$;
in general this is a very difficult problem. A \emph{Marstrand-type} Theorem deals not with a
fixed foliation but rather with almost every foliation in a given foliations space
endowed with some version of the Lebesgue measure.

We will usually make an abuse of language and speak of a foliation $\mathcal F$ of a space $X$
when in fact something has to be removed from both $X$ and $\mathcal F$ in order to get a genuine partition. 
For example we may say ``look at the foliation of the plane by lines passing through the origin''. What we mean
in this case is ``look at the foliation of the plane minus the origin by lines passing through the origin with
the origin removed''. Likewise, if some line $L$ is fixed in $\R^3$, the family of all affine planes containing $L$, 
with $L$ removed, is a foliation of $\R^3 \setminus L$, but we will actually write ``Let $\mathcal F$ be the
foliation of $\R^3$ by affine planes containing $L$''. It would be very unpleasant and quite pedantic to write down in 
every case which subset should be removed from the space and the leaves, and we leave it to the reader to make the obvious
corrections.

\subsubsection{Transversality and Kaufman's argument}
Transversality of a family of foliations (or a family of projections) is the crucial property to look for 
when one sets out to prove a Marstrand-type result. In the presence of transversality, a quite general argument, 
due to Kaufman \cite{Kaufman1975}, allows to prove a version of Marstrand's projection Theorem, as well as some 
improvements which we do not discuss in this paper in order to keep things short. 

In this paper, we are not going to improve on Kaufman's argument. Our purpose is to introduce ``good foliations''
and transversality will follow quite naturally from the definition (and the product formula, see \ref{ss.product-formula}).

A general exposition of Kaufman's argument in an abstract setting (dealing with parametrized families of projections)
can be found in \cite{Peres2000}. We write down a detailed proof of Corollary \ref{cor.first-transversality} because 
there are some issues, requiring us to work ``locally'' (cutting the measure 
into small pieces), that do not appear in Kaufman's usual argument. 

Our statements deal with dimension of sets; it would be equally sensible to concern ourselves with dimension of measures
and in the proofs this is what we actually do, implicitly using Frostman's Lemma, as in \cite{Mattila1995}.

\subsection{Hermitian forms on the Grassmann algebra}\label{ss.hermitian-forms}

\subsubsection{The Grassmann algebra}\label{sss.grassmann-algebra}

We refer to \cite{Barnabei1985} for a good elementary exposition of the Grassmann exterior (bi)algebra
associated to a vector space.
Another good reference is \cite{Bourbaki1998}, but beware of the conflicting notations.
In this paper we will use the same notations as in \cite{Barnabei1985}.

We now recall some basic definitions and fix appropriate notations.

Let $\K$ be $\R$ or $\C$ and fix a finite-dimensional $\K$-vector space $E$. The
elements of $E$ will usually be denoted by the letters $u, v, w$.

We denote by $\GB_\K(E)$\label{def.GB} the Grassmann algebra (over $\K$) associated with
$E$. The Grassmann algebra is also called the exterior algebra of $E$. The
\emph{progressive product} (also called \emph{exterior product}) will be denoted by the
symbol $\vee$\label{def.vee}
The \emph{regressive product} (to be introduced later) will be denoted
by $\wedge$.

\begin{remark}
	If $\K$ is $\C$, we may consider the Grassmann algebra over $\R$,
	$\GB_\R(E)$, as well as the Grassmann algebra over $\C$, $\GB_\C(E)$.
	In this paper we will always work with the Grassmann algebra over $\C$ (when $\K = \C$).

	In a later paper, we will also consider the Grassmann algebra, over $\R$,
	of a complex space, and this will allow us to define and study a family
	of foliations (called \emph{real spheres} or \emph{Ptolemy circles}) different from the
	\emph{chains} which are the main focus of this paper. See \emph{e.g.} \cite{Goldman1999}
	for definitions.

	Henceforth, we drop the subscript $\K$ in the notation for the Grassmann
	algebra.
\end{remark}

The subspace of $k$--vectors will be denoted by $\GB^k (E)$\label{def.GBk}. If $n$ is the
dimension of $E$ (over $\K$),
\begin{equation}
	\GB(E) = \bigoplus_{k=0}^n \GB^k(E)
\end{equation}
where $\GB^0(E)$ is \emph{canonically isomorphic} to $\K$, $\GB^1(E)$ is
\emph{canonically isomorphic} to $E$, $\GB^{n-1}(E)$ is \emph{non-canonically isomorphic} to the algebraic dual of
$E$, and $\GB^n(E)$ is \emph{non-canonically isomorphic} to $\K$ (the choice of a basis of
$\GB^n(E)$ is equivalent to the choice of a non-degenerate alternating $n$--linear form).

The above direct sum is \emph{graded}: for $U \in \GB^k(E)$ and $V \in \GB^\ell(E)$,
the progressive product $U \vee V$ belongs to $\GB^{k+\ell}(E)$ (where by definition
$\GB^i(E) = 0$ if $i > n$).

The set of \emph{pure} or \emph{decomposable} $k$--vectors, i.e. $k$--vectors of the form $u_1 \vee \cdots \vee u_k$ ($u_1,\ldots, u_k \in E$)
will be denoted by $\DB^k(E)$\label{def.DBk}. We will often use capital letters $U, V, W$ to denote $k$--vectors, and most of the time we
will consider pure $k$--vectors only.

Bear in mind that unless $k = 1$ or $k = \dim(E) - 1$, $\DB^k(E)$ is not a
vector space.

If $U = u_1 \vee \cdots \vee u_k$ is a non-zero element of $\DB^k(E)$, we denote by
$\Span(U)$\label{def.span} the vector subspace $\K u_1\oplus \cdots \oplus  \K u_k$ of $E$. This is the smallest
vector subspace $E'$ of $E$ such that $U$ belongs to $\GB^k(E')$.

Thus, if $U$ and $V$ are, respectively, a pure $k$--vector and a pure $\ell$--vector,
such that $U \vee V \neq 0$, then $\Span(U \vee V ) = \Span(U) \oplus \Span(V)$.

The basic fact that $k$--dimensional vector subspaces of $E$ are in one-to-one
correspondance with projective classes of elements of $\DB^{k+1}(E)$ will
be used at every moment throughout this paper.

\subsubsection{The regressive product}

We now recall briefly the definition of the \emph{regressive product}\label{def.wedge}. Let $n$ be
the dimension of $E$ (over $\K$). The choice of a non-degenerate alternating
$n$-linear form $\omega$ on $E$ yields a \emph{Hodge isomorphism}
\begin{equation}
	* : \GB(E) \to \GB(E^*)
\end{equation}
that identifies $\GB^k(E)$ with $\GB^{n-k}(E^*)$ (were $E^*$ is the algebraic dual space
of E in the usual sense).

The pull-back, through this isomorphism, of the progressive product
in $\GB(E^*)$ is, by definition, the regressive product in $\GB(E)$, denoted by $\wedge$.

The regressive product depends on the choice of $\omega$; to put it differently,
it depends on the choice of a basis of the 1--dimensional space $\GB^n(E)$.

By definition,
\begin{equation}
	U \wedge V = (U^* \vee V^*)^*
\end{equation}

Let us also recall the definition of the Hodge isomorphism. The bilinear mapping
\begin{equation}
	\GB^k(E) \times \GB^{n-k}(E) \to \GB^n(E)
\end{equation}

defined by
\begin{equation}
	(u_1 \vee \cdots \vee u_k, u_{k+1} \vee \cdots \vee u_n) \mapsto u_1 \vee \cdots \vee u_n
\end{equation}
(and extended by linearity) is composed with the isomorphism
$\GB^n(E) \to \K$ associated to $\omega$,
\begin{equation}
	u_1 \vee \cdots \vee u_n \mapsto \omega(u_1, \ldots, u_n)
\end{equation}

and identifies $\GB^k(E)$ with the dual of $\GB^{n-k}(E)$; this dual is also canonically
isomorphic to $\GB^{n-k}(E^*)$. In this way, we obtain for every $k$ an
isomorphism $\GB^k(E) \to \GB^{n-k}(E^*)$ which is, by definition, the Hodge isomorphism restricted 
to $\GB^k(E)$.

For details, see \cite{Bourbaki1998} or \cite{Barnabei1985}.

The geometric significance of the regressive product should be clear:
if $U \wedge V \neq 0$, where $U$, $V$ are, respectively, a pure $k$--vector and a pure $\ell$--vector, 
and $k+\ell \geq n$, then $U \wedge V$ is a pure $(k+\ell-n)$--vector such that 
$\Span(U \wedge V) = \Span(U) \cap \Span(V)$; if $k+\ell <n$, $U \wedge V=0$.

Endowed with $\vee$ and $\wedge$, $\GB(E)$ is the Grassmann bialgebra of $E$.

\subsubsection{Grassmann extensions of Hermitian forms}\label{sss.grassmann-extensions}

Let $E$ be, as before, a finite-dimensional $\K$-vector space ($\K = \R$ or
$\C$), now endowed with a sesquilinear form $\Phi$; by definition $\Phi(\alpha u, \beta v) =
	\overline{\alpha} \beta \Phi(u, v)$ for $\alpha, \beta  \in \K$ and $u, v \in E$.
If $\K$ is $\R$ this is bilinearity in the usual sense.

Denote by $\overline{E}$ the $\K$--vector space with the same underlying additive
group as $E$ and the $\K$--operation law defined by $\alpha \cdot u= \overline{\alpha} u$
(where the right-hand side denotes the operation of $\alpha$ on $u$ in $E$).

If $\K$ is $\R$, $\overline E$ is equal to $E$, whereas if $\K = \C$ the identity mapping is
an anti-isomorphism $\overline E \to E$.

We now recall, as in \cite{Bourbaki2007}, the canonical extension of the sesquilinear form $\Phi$
to the Grassmann algebra $\GB(E)$.

Fix $k \geq 1$. The mapping $\overline{E}^k \times E^k \to \K$ defined by
\begin{equation}
	(u_1, \ldots, u_k; v_1, \ldots, v_k) \mapsto \Det (\Phi(u_i, v_j ))
\end{equation}

(where $\Det$ is the usual determinant of a $k \times k$ matrix) is $\K$--multilinear.

Since, also, the right-hand side is zero as soon as $u_i = u_j$ or $v_i = v_j$ for some $i \neq j$,
this mapping yields, by the universal property of the Grassmann algebra
(see \cite{Bourbaki1998}, \S 7, Proposition 7), a $\K$--bilinear form $\GB^k(\overline{E}) \times \GB^k(E) \to \K$.

Using the canonical antilinear identication of $\GB^k(\overline{E})$ with $\overline{\GB^k(E)}$, we
obtain a sesquilinear form on $\GB^k(E)$ which we denote by $\GB^k(\Phi)$\label{def.GBkPhi}.

If $\Phi$ is Hermitian, meaning that $\Phi(v, u) = \overline{\Phi(u, v)}$, so is $\GB^k(\Phi)$. Also if $\Phi$
is non-degenerate, $\GB^k(\Phi)$ is non-degenerate as well; if $\Phi$ is definite, $\GB^k(\Phi)$
is definite, and of the same sign.

The following result is basic.

\begin{lemma} \label{lemma.basic}
	If $U_1,U_2 \in \DB^k(E)$, and $V \in \DB^\ell(E)$, are such that $\Span(U_1)$ and $\Span(U_2)$ are both $\Phi$--orthogonal
	to $\Span(V)$,
	\begin{equation}
		\GB^{k+\ell} (\Phi) (U_1 \vee V, U_2 \vee V) = \GB^k (U_1,U_2) \times \GB^\ell (V,V)
	\end{equation}
\end{lemma}

\subsubsection{Basic properties of the Hermitian norm}\label{ss.basic-properties}
Fix $n \geq 1$ and let again $\K$  be $\R$ or $\C$. We denote the canonical basis 
of $\K^{n+1}$ by $(e_0,\ldots,e_n)$. For any $u, v \in \K^{n+1}$, we denote
the usual Hermitian inner product by
\begin{equation}\label{def.inner-product}
	\inprod{u}{v} = \sum_{i=0}^{n} \overline{x_i} y_i
\end{equation}
(where $u=(x_0,\ldots,x_n)$ and $v=(y_0,\ldots,y_n)$)
and we use the same symbol for the canonical Grassmann extension, i.e.
\begin{equation}\label{def.extended-inner-product}
	\inprod{u_1 \vee \cdots \vee u_k}{v_1 \vee \cdots \vee v_k} = \Det (\inprod{u_i}{v_j})
\end{equation}
(where the right-hand side is the determinant of the $k \times k$ matrix whose $(i, j)$--component
is $\inprod{u_i}{v_j}$. This Grassmann extension is still a Hermitian
inner product on $\GB^k (\K^{n+1})$ and the associated Hermitian norm is denoted, as usual,
by $\| \cdot \|$.

The $n$--dimensional projective space over $\K$  is denoted by $\PB_\K^n$\label{def.projective-space} ; this is
the space of $\K$--vector lines in $\K^{n+1}$. If $\K  = \R$, respectively $\K  = \C$, $\PB^n_\K$ is a
Riemannian manifold of dimension $n$, respectively a Hermitian manifold
of complex dimension $n$ (and real dimension $2n$). In general, if $E$ is some $\K$--vector space,
the symbol $\PB E$\label{def.PB} denotes the projective space associated to $E$ over $\K$.

We will also use the notation $\PB \DB^k (\K^{n+1})$\label{def.PBDBK} to denote the space of projective classes of
elements of $\DB^k (\K^{n+1})$. (Note that $\DB^k (\K^{n+1})$ is not a vector space in general.)

In this paper, the letter $d$ will always denote the \emph{angular metric} on $\PB^n_\K$ defined by
\begin{equation}\label{eq.angular-metric}
	d(u,v) = \Dist{u}{v}
\end{equation}
where $u, v$ are non-zero elements of $\K^{n+1}$. In the left-hand side we are abusing notations and denoting elements of
$\PB^n_\K$ by corresponding elements of $\K^{n+1}$. It seems preferable to slightly abuse notations rather than
use the cumbersome notation $[u],[v]$ when we are dealing with projective classes.

From this definition, the following formula follows at once:
\begin{equation}\label{eq.basic-formula}
	d(u,v)^2 = 1 - \frac{|\inprod{u}{v}|^2}{\|u\|^2 \cdot \|v\|^2} = \sin^2 (\theta)
\end{equation}
where $\theta$ is the (non-oriented) angle from $u$ to $v$.

The orthogonal complement with respect to the Hermitian inner product will be denoted by $\perp$\label{def.perp}.
For example, $v^\perp $ is the space of all vectors
$u \in \K^{n+1}$ such that $\inprod{u}{v} = 0$.

\begin{lemma}\label{lemma.basic-inequality}
	For any $U \in \DB^k(\K^{n+1}), V \in \DB^\ell(\K^{n+1})$,
	\begin{equation}
		\| U \vee V \| \leq \| U \| \cdot \| V \|
	\end{equation}
\end{lemma}
and this is an equality if and only if $\Span(U)$ and $\Span(V)$ are orthogonal.

This follows from the following Lemma which we state separately for
future reference.

\begin{lemma}\label{lemma.basic-orthogonality}
	Let $V \in \DB^\ell(\K^{n+1})$ and denote
	\begin{itemize}
		\item $\pi_V^\perp$ the orthogonal projection $\K^{n+1} \to \Span(V)^\perp$
		\item $\Pi_V^\perp$ the orthogonal projection $\GB^k (\K^{n+1}) \to \GB^k(\Span(V)^\perp)$.
	\end{itemize}
	Then $\GB^k (\pi_V^\perp)=\Pi_V^\perp$ and for any $U \in \DB^k (\K^{n+1})$,
	\begin{equation}
		\| U \vee V \| = \| \Pi_V^\perp (U) \| \cdot \| V \|
	\end{equation}
\end{lemma}

The notation $\GB^k(f)$, where $f$ is a linear mapping with domain $E$,
stands for the extension of $f$ to $\GB^k(E)$, which is characterized by the
relation $\GB^k(f)(u_1 \vee \cdots \vee u_k) = f(u_1) \vee \cdots \vee f(u_k)$ for any $u_1, \ldots, u_k \in E$.
\begin{proof}
	Recall the basic property that a linear projection $\pi$ is orthogonal if and only if
	for any vector $x$ in the image of $\pi$ and any other vector $y$, $\inprod{x}{\pi(y)}$ is equal to $\inprod{x}{y}$.

	To show that $\GB^k (\pi_V^\perp)$ is the orthogonal projection onto $\GB^k (\Span(V)^\perp)$, it is enough
	to check that for any $U \in \DB^k (\K^{n+1})$ and any $U' \in \DB^k (\Span(V)^\perp)$,
	\begin{equation}
		\inprod{U}{U'} = \inprod{\GB^k (\pi_V^\perp) (U)}{U'}
	\end{equation}
	Now by definition the right-hand side is equal to $\Det(\inprod{\pi_V^\perp(u_i)}{u_j'})$ and $\inprod{\pi_V^\perp(u_i)}{u_j'}=\inprod{u_i}{u_j'}$
	by the basic property of orthogonal projections; this determinant is thus equal to $\Det(\inprod{u_i}{u_j'}) = \inprod{U}{U'}$.

	The formula for norms then follows from the fact that $U \vee V = \GB^k (\pi_V^\perp) (U) \vee V$ by definition of the progressive product.
\end{proof}

\subsubsection{First distance formula}\label{ss.first-distance-formula}

\begin{theorem}\label{th.first-distance-formula}
	Let $U= u_0 \vee \cdots \vee u_k$ be a non-zero element of $\DB^{k+1}(\K^{n+1})$ and let
	$w \in \K^{n+1}$ be non-zero. The quantity
	\begin{equation}
		\tau(U,w)=  \Dist{U}{w}
	\end{equation}
	is equal to the distance between $w$ and the $k$--dimensional projective subspace $\PB_\K (\Span(U))$ in $\PB_\K^{n+1}$.
\end{theorem}
We commit the usual abuse of language of denoting by $w$ both a non-zero vector and its image in $\PB_\K^n$.
\begin{proof}
	Let $\pi$ be the orthogonal projection from $\K^{n+1}$ onto $\Span(U)$. Then $w - \pi(w)$ is orthogonal to $\Span(U)$,
	so $\tau(U,w)^2$ is equal  to
	\begin{equation}
		\frac{\| w - \pi(w)\|^2}{\|w\|^2} = 1 - \frac{\| \pi(w) \|^2}{\|w\|^2} = 1 - \sup_{v \in \Span(U)} \frac{| \inprod{w}{v} |^2}{\|w\|^2 \cdot \|v\|^2}
	\end{equation}
	(where we used the convexity of orthogonal projections and ). By using formula \eqref{eq.basic-formula}, we see
	 that the right-hand side of the last equation is equal to
	\begin{equation}
		\inf_{v \in \Span(U)} d(w,v)^2 = d(w,\PB (\Span(U)))^2
	\end{equation}

\end{proof}

\section{Linear foliations in real and complex projective spaces}\label{s.linear-foliations}

In this section, we fix an integer $n \geq 2$ and we work in the $n$--dimensional
$\K$--projective space $\PB_\K^n$.

If $U = u_0 \vee \cdots \vee u_k$ is a non-zero element of $\DB^{k+1}(\K^{n+1})$ ($k \leq n$), we will
denote by $L_U$ the projective subspace $\PB(\Span(U))$ of $\PB_\K^n$. The mapping
$[U] \mapsto L_U$ is a bijection from $\PB \DB^{k+1}(\K^{n+1})$ to the space of $k$--dimensional
$\K$--projective subspaces of $\PB^n_\K$. We will identify these spaces and say ``let
$[u_0 \vee \cdots \vee u_k]$ be a $k$--dimensional projective subspace of $\PB^n_\K$.''

\subsection{Generalized radial foliations}\label{ss.radial-foliations}
Fix an integer $k$, $0 \leq k \leq n-2$. For any $k$--dimensional projective subspace
$L$ of $\PB^n_\K$, and any $x \in \PB^n_\K \setminus L$, there is one and only one $(k+1)$--dimensional
projective subspace of $\PB^n_\K$ containing $L \cup \{x\}$.

To $L$ we may thus associate a foliation of $\PB^n_\K$ by $(k+1)$--dimensional
projective subspaces. We exclude the case $k = n - 1$ because the foliation
is then trivial (it has only one leaf).

(Recall that when we say ``a foliation of $\PB^n_\K$  by projective subspaces'' here we
actually mean ``a foliation of $\PB^n_\K \setminus L_U$ by projective subspaces containing $L_U$, 
with $L_U$ removed''.)

Let us describe this foliation algebraically, thanks to the Grassmann
algebra, in order to perform computations.

Fix a $k$--dimensional projective subspace $[U] = [u_0\vee\cdots\vee u_k] \in \PB \DB^{k+1}(\K^{n+1})$.
of $\PB^n_\K$ and denote by $\Proj_U$\label{def.projU}

\begin{equation}
	\begin{array}{lccc}
		\Proj_U: & \PB_\K^n \setminus L_U & \to     & \PB \DB^{k+1}(\K^{n+1})           \\
		         & [w]                    & \mapsto & [u_0 \vee \cdots \vee u_k \vee w] \\
	\end{array}
\end{equation}

By definition, the fibers of this mapping are exactly the $(k + 1)$--dimensional projective subspaces of $\PB_\K^n$ containing $L_U$.

We are now going to endow $\PB \DB^{k+1}(\K^{n+1})$ with a natural metric, in order to be able
to prove the needed transversality properties for our linear foliations.

\subsection{The angular metric on the codomain}\label{ss.codomain-metric}

We endowed earlier $\GB^\ell (\K^{n+1})$ with a Hermitian structure for $0 \leq \ell \leq n+1$. To this Hermitian space
we may associate its degree 2 Grassmann algebra, $\GB^2(\GB^\ell(\K^{n+1}))$ and this is in turn a Hermitian space in a natural
way. We are now looking at the Grassman algebra arising from the vector space underlying a Grassmann algebra;
its elements are of the form $U \vee V$ , where $U$ and $V$ belong
to $\GB^\ell(\K^{n+1})$, and the reader should be careful not to believe that, somehow,
if $U = u_1 \vee \cdots \vee u_\ell$ and $V = v_1 \vee \cdots \vee v_\ell$, the element $U \vee V$ of
$\GB^2(\GB^\ell(\K^{n+1}))$ could be equal to the element $U \vee V$ of $\GB^{2 \ell}(\K^{n+1})$. These
elements do not sit in the same space. They are the same thing if and
only if $\ell = 1$.

This construction allows to endow $\PB \GB^\ell(\K^{n+1})$  (the projective space associated to 
$\GB^\ell (\K^{n+1}))$ with the metric defined as in \eqref{eq.angular-metric}.

In turn, the restriction of this metric to $\PB \DB^\ell(\K^{n+1})$ endows the space
of $(\ell - 1)$-projective subspaces of $\PB_\K^n$ with a natural metric.

Our aim in the paragraph to follow is to study the distance
between $\Proj_U(w_1)$ and $\Proj_U(w_2)$ for $w_1,w_2 \in \PB_\K^n$.

\subsection{The product formula}\label{ss.product-formula}
\begin{theorem}\label{th.product-formula}
	Let $p \geq 1$. For any element $V \in \GB^{p-1} (\K^{n+1})$ and any $w_1,w_2 \in \PB_\K^n \setminus L_V$,
	\begin{equation}
		\| (V \vee w_1) \vee (V \vee w_2) \| = \| V \| \cdot \| V \vee w_1 \vee w_2 \| \label{eq.product-formula}
	\end{equation}
	where $(V \vee w_1) \vee (V \vee w_2)$ belongs to $\GB^2 (\GB^p (\K^{n+1}))$ and $V \vee w_1 \vee w_2$ belongs to $\GB^{p+1}(\K^{n+1})$.
\end{theorem}
\begin{proof}
	Denote by $\pi$ the orthogonal projection $\K^{n+1} \to \Span(V)^\perp$. By the basic properties of progressive product
	we have
	\begin{equation}
		V \vee w_1 = V \vee \pi(w_1)
	\end{equation}
	and similarly for $w_2$; we also have, for the same reasons,
	\begin{equation}
		V \vee w_1 \vee w_2 = V \vee \pi(w_1) \vee \pi(w_2)
	\end{equation}
	Without loss of generality, we can thus assume that $w_1$ and $w_2$ are orthogonal to $\Span(V)$.

	Now, by definition, the square of the left-hand side in equation \eqref{eq.product-formula} is equal to the $2 \times 2$ determinant
	\begin{equation}
		\left| \begin{array}{cc}
			\|V \vee w_1 \|^2               & \inprod{V \vee w_1}{V \vee w_2} \\
			\inprod{V \vee w_2}{V \vee w_1} & \| V \vee w_2\|^2\end{array} \right|
	\end{equation}
	We can apply Lemma \ref{lemma.basic} and the previous determinant is equal to
	\begin{equation}
		\|V\|^4 \cdot \left| \begin{array}{cc} \|w_1\|^2    & \inprod{w_1}{w_2} \\
			\inprod{w_2}{w_1} & \| w_2 \|^2\end{array} \right| = \|V\|^4 \cdot \| w_1 \vee w_2 \|^2
	\end{equation}
	On the other hand, using again orthogonality of $w_1$ and $w_2$ with respect to $\Span(V)$,
	as well as Lemma \ref{lemma.basic} we see that
	\begin{equation}
		\| V \vee w_1 \vee w_2 \| = \| V \| \cdot \|w_1 \vee w_2 \|
	\end{equation}
	and the proof is over.
\end{proof}

For any elements $U \in \DB^k (\K^{n+1})$, $V \in \DB^\ell (\K^{n+1})$, we denote by $\tau(U,V)$ the number
\begin{equation} \label{def.tau}
	\tau(U,V) = \Dist{U}{V}
\end{equation}
where $U \vee V \in \DB^{k+\ell}(\K^{n+1})$. If $k$ (resp. $\ell$) is equal to $1$, this is the distance from
$U$ (resp. $V$) to $\PB (\Span(V))$ (resp. $\PB (\Span(U))$) (Theorem \ref{th.first-distance-formula}). Also, note that $\tau(U,V) \leq 1$.

With this notation, the previous Theorem has the following consequence:

\begin{equation}\label{eq.proj-dist-formula}
	d (\Proj_U (w_1),\Proj_U(w_2)) = \frac{\tau(U,w_1 \vee w_2)}{\tau(U,w_1) \tau(U,w_2)} d(w_1,w_2)
\end{equation}
which will play a crucial role in our analysis.

\subsection{Generalized distance formula}\label{ss.generalized-distance}
In this paragraph, we elucidate the geometric significance of the number
$\tau(U,V)$ introduced above; this is not needed in the rest of the paper.

We start with some notations and a lemma. Fix $q \geq 2$ and $k,\ell \geq 1$.
such that $k + \ell \leq q$. If $V$ is some non-zero decomposable $\ell$-vector of $\K^q$,
i.e. $V \in \DB^\ell (\K^q)$, we let
\begin{equation}
	\GB_0^k (V;\K^q) = \{ U \in \GB^k (\K^q)\ ;\ U \vee V = 0 \}
\end{equation}

This is the \emph{annihilator of $V$ in $\GB^k (\K^q)$}, a vector subspace of $\GB^k (\K^q)$.

\begin{lemma}
	The orthogonal complement of $\GB^k (\Span(V)^\perp)$ in $\GB^k (\K^q)$ is equal to $\GB_0^k (V;\K^q)$.
\end{lemma}
\begin{proof}
	Let $\pi$ be the orthogonal projection $\K^q \to \Span(V)^\perp$. Then $\GB^k (\pi)$
	is the orthogonal projection $\GB^k (\K^q) \to \GB^k(\Span(V)^\perp)$. Here, we denote by
	$\GB^k(\pi)$ the extension of $\pi$ to $\GB^k (\K^q)$.

	Also, it follows from the definition of progressive product that $U \vee V = \GB^k (\pi) (U) \vee V$.
	This $(k+\ell)$-vector is zero if and only if $\GB^k(\pi) (U)=0$, which is equivalent to saying
	that $U$ is orthogonal to $\GB^k (\Span(V)^\perp)$.
\end{proof}

\begin{theorem} \label{th.generalized-distance}
	For $U \in \DB^k (\K^q)$ and $V \in \DB^\ell (\K^q)$,
	\begin{equation}
		\Dist{U}{V} = d(U,\PB \GB_0^k(V;\K^q))
	\end{equation}
	where $U \vee V$ belongs to $\DB^{k+\ell}(\K^q)$.
\end{theorem}
In other words, the number $\tau(U,V)$ is equal to the distance, in $\PB \GB^k (\K^q)$,
between $U$ and the projective subspace $\PB \GB_0^k (V;\K^q)$.
\begin{proof}
	Let $\pi$ be the orthogonal projection from $\K^q$ onto $\Span(V)^\perp$.
	We compute, taking into account the previous Lemma,
	\begin{equation}
		1 - \left( \Dist{U}{V} \right)^2 = 1 - \frac{\| \GB^k (\pi) (U) \|^2}{\|U\|^2} =
		\frac{\|U-\GB^k (\pi)(U)\|^2}{\|U\|^2}
	\end{equation}
	and $U - \GB^k (\pi)(U)$ is the image of $U$ through the orthogonal projection onto
	$\GB^k (\Span(V)^\perp)^\perp$. The convexity property of inner product then yields
	\begin{equation}
		\frac{\|U-\GB^k (\pi)(U)\|^2}{\|U\|^2} = \sup_{W} \frac{|\inprod{U}{W}|^2}{\|U\|^2 \cdot \|W\|^2}
	\end{equation}
	where the supremum is relative to $W \in \GB^k(\Span(V)^\perp)^\perp = \GB_0^k (V;\K^q)$. The right-hand side is equal to
	\begin{equation}
		1 - \inf_W \left( \Dist{U}{W} \right)^2 = 1 - \inf_W d(U,W)^2
	\end{equation}
	(where still $W \in \GB^k_0 (V;\K^q)$ and $U \vee W \in \GB^2 (\GB^k (\K^q))$) and the Lemma is proved.
\end{proof}

\subsection{Lipschitz property; transversality; Marstrand-type Theorem}\label{ss.first-transversality}

Recall the setting from the beginning of the section: $n \geq 2$ is fixed, $0 \leq k \leq n-2$
and $U=u_0 \vee \cdots \vee u_k \in \DB^{k+1}(\K^{n+1})$ is a (momentarily fixed)
decomposable $(k+1)$--vector of $\K^{n+1}$.

Introduce the Lipschitz modulus function
\begin{equation} \label{def.phiU}
	\phi_U (w_1,w_2) = \frac{d(\Proj_U(w_1),\Proj_U(w_2))}{d(w_1,w_2)}
\end{equation}
for any pair of distinct $w_1,w_2 \in \PB_\K^n \setminus L_U$.

Recall the formula \eqref{eq.proj-dist-formula}
\begin{equation} \label{eq.phiU-formula}
	\phi_U(w_1,w_2) = \frac{\tau(U,w_1 \vee w_2)}{\tau(U,w_1) \tau(U,w_2)}
\end{equation}

A basic fact is that $\Proj_U$ is ``locally Lipschitz''.
\begin{proposition} \label{prop.lipschitz}
	The restriction of $\Proj_U$ to any compact subspace of $\PB_\K^n \setminus L_U$
	enjoys the Lipschitz property.
\end{proposition}
\begin{proof}
	By the above formula \eqref{eq.phiU-formula}, the Proposition follows from the fact
	that the function $w \mapsto \tau(w,U)$ is continuous and non-zero
	in $\PB_\K^n \setminus L_U$.
\end{proof}

The space $\PB \DB^{k+1} (\K^{n+})$ is a Hermitian manifold and carries a natural Lebesgue
measure $\Leb$\label{def.Leb}. This measure can be defined in an elementary way: endow
$(\PB_\K^n)^{k+1} = \PB_\K^n \times \cdots \times \PB_\K^n$ with the product ($k+1$ times)
of the Lebesgue measure of $\PB_\K^n$, and push this product measure forward through the
almost-everywhere defined mapping
\begin{equation}
	\begin{array}{rcl}
		(\PB_\K^n)^{k+1} = \PB_\K^n \times \cdots \times \PB_\K^n & \to     &
		\PB \DB^{k+1} (\K^{n+1})                                                                         \\
		([u_0],\ldots,[u_k])                                      & \mapsto & [u_0 \vee \cdots \vee u_k] \\\end{array}
\end{equation}

We now state our first transversality result.
\begin{proposition} \label{prop.first-transversality}
	For any distinct $w_1,w_2 \in \PB_\K^n$ and any $r>0$,
	\begin{equation}
		\Leb \{ U \in \PB \DB^{k+1} (\K^{n+1})\ ;\ \phi_U(w_1,w_2) < r \} \lesssim r^{\delta_\K (n-k-1)}
	\end{equation}
	uniformly in $w_1,w_2$, where $\delta_\K$ is $1$ if $\K=\R$ and $2$ if $\K=\C$.
\end{proposition}
\begin{proof}
	Since $\phi_U(w_1,w_2) \geq \tau(U,W)$ by \eqref{eq.phiU-formula}, where we let $W=w_1 \vee w_2$, it is enough
	to show that
	\begin{equation}
		\Leb \{ U \in \PB \DB^{k+1} (\K^{n+1}) \ ;\ \tau(U,W) < r \} \lesssim r^{n-k-1}
	\end{equation}
	If $k=0$ this is a special case of Lemma \ref{lemma.first-transversality} below. If $k \geq 1$ we argue by induction
	using the inequality
	\begin{equation}
		\tau(U,W) \geq \tau(u_k,U' \vee W) \tau(U',W)
	\end{equation}
	where $U=u_0 \vee \cdots \vee u_k$ and $U' = u_0 \vee \cdots \vee u_{k-1}$. Lemma \ref{lemma.first-transversality}
	along with Fubini's Theorem yield
	\begin{multline}
		\Leb \{ U \in \PB \DB^{k+1} (\K^{n+1})\ ;\ \tau(U,W) < r \}\\ \lesssim
		r^{\delta_\K (n-k-1)} \int \mathrm{d} \Leb(U') \ \tau(U',W)^{-\delta_\K (n-k-1)}
	\end{multline}

	The induction hypothesis implies that $\int \mathrm{d} \Leb(U') \ \tau(U',W)^{-\delta_\K (n-k-1)}$ is finite,
	and the Proposition follows.
\end{proof}

\begin{lemma}\label{lemma.first-transversality}
	For any $\ell$-dimensional projective subspace $L$ of $\PB_\K^n$
	\begin{equation}
		\Leb \{ u \in \PB_\K^n\ ;\ d(u,L) \leq r\} \lesssim r^{\delta_\K (n-\ell)}
	\end{equation}
	uniformly in $r$.
\end{lemma}
This Lemma and its proof are standard.

\begin{corollary} \label{cor.first-transversality}
	Fix $k$, $0 \leq k \leq n-2$. Let $A$ be a Borel subset of $\PB_\K^n$ of Hausdorff dimension $s$. For almost
	every $k$--dimensional projective subspace $L$ of $\PB_\K^n$, the transverse dimension of $A$, with respect to
	the foliation of $\PB_\K^n \setminus L$ by $(k+1)$-dimensional projective subspace containing $L$, is equal
	to
	\begin{equation}
		\inf \{\delta_\K(n-k-1),s\}
	\end{equation}
\end{corollary}
\begin{proof}
	We apply Kaufman's classical argument using the transversality property stated in Proposition
	\ref{prop.first-transversality}. Let us provide some details. We assume $s>0$.

	First, note that if $v_1,v_2$ are two different points of $\PB_\K^n$, the set of $k$--dimensional 
	projective subspaces passing through $v_1$ and $v_2$ has Lebesgue measure $0$ in $\PB \DB^{k+1}(\K^{n+1})$.

	Taking this fact into account, pick $2$  disjoint closed balls
	$B_1, B_2$ such that the Hausdorff dimension of $A \cap B_i$ is $s$ for each $i$,
	and small enough that the set of $k$--dimensional projective subspace of $\PB^n_\K$ meeting both $B_1$ and $B_2$
	has very small Lebesgue measure. This is possible because the Hausdorff dimension of a finite union 
	$\cup X_i$ is the supremum of the Hausdorff dimensions of the $X_i$. 

	Let $O_1,O_2$ be open subsets of $\PB \DB^{k+1}(\K^{n+1})$ such that for any $U \in \overline{O_i}$, the projective subspace
	$L_U$ does not meet $B_i$, and that the complement of $O_1 \cup O_2$ has very small Lebesgue measure in
	$\PB \DB^{k+1}(\K^{n+1})$. Now fix $i=1$ or $2$.
	
	Let $\sigma < \inf \{ s, \delta_\K (n-k-1)\}$ and $\mu$ be a Borel probability measure supported on $A \cap B_i$
	such that the $\sigma$--energy of $\mu$ is finite:
	\begin{equation} \label{def.energy}
		I_{\sigma} (\mu) = \int \frac{\mathrm{d} \mu (w_1) \mathrm{d} \mu(w_2)}{d(w_1,w_2)^{\sigma}} < \infty 
	\end{equation}
	(see \cite{Mattila1995} 8.8 and 8.9).
	We apply Fubini's Theorem:
	\begin{equation}
		\int_{O_i} \mathrm{d} \Leb (U)  I_{\sigma} (\Proj_U (\mu)) = 
		\int \frac{\mathrm{d} \mu (w_1) \mathrm{d} \mu(w_2)}{d(w_1,w_2)^{\sigma}} \int_{O_i} \mathrm{d} \Leb (U)
		\phi_U (w_1,w_2)^{-\sigma}
	\end{equation}
	and we will show that the right-hand side is finite by checking that 
	\begin{equation}
		\int_O \mathrm{d} \Leb(U) \phi_U(w_1,w_2)^{-\sigma}
	\end{equation}
	is bounded by a uniform constant for any distinct $w_1,w_2 \in  B_i$.

	A standard application of Fubini's Theorem followed by a change of variable yields 
	\begin{equation}
		\int_{O_i} \mathrm{d} \Leb(U) \phi_U(w_1,w_2)^{-\sigma} \lesssim 1+\int_0^1 \Leb\{ U \in O_i\ ;\ \phi_U(w_1,w_2) < t\} t^{-(1+\sigma)} \ \mathrm{d}t
	\end{equation}
	(where the constant implied by the notation $\lesssim$ does not depend on $w_1,w_2$). Taking into account 
	Proposition \ref{prop.first-transversality}, we see that the right-hand side is bounded by a uniform constant as soon as 
	$\sigma < \delta_\K (n-k-1)$, which holds by assumption. 

	All in all, we get 
	\begin{equation}
		\int_{O_i} \mathrm{d} \Leb (U)  I_{\sigma} (\Proj_U (\mu)) \lesssim I_\sigma (\mu) < \infty 
	\end{equation}
	showing that for Lebesgue--almost every $U \in O_i$, the transverse dimension of $A \cap B_i$, along the foliation 
	by $(k+1)$--dimensional projective subspaces containing $U$, is at least equal to $\sigma-\varepsilon$. 
	
	Thus for almost every $U \in O_1 \cup O_2$, the transverse dimension of $A$, along the foliation by 
	$(k+1)$--dimensional projective subspaces containing $U$, is at least equal to $\sigma-\varepsilon$. 
	Since $\varepsilon$ was arbitrary, we get the desired conclusion for almost every $U \in O_1 \cup O_2$.

	The Theorem follows from this, because the complement of $O_1 \cup O_2$ has arbitrarily small Lebesgue measure.
\end{proof}

The previous results will now be improved by looking at a special subfamily of foliations.

\subsection{Transversality of pointed foliations}\label{ss.transversality-pointed}
As before, $n \geq 2$ is fixed and we work in $\PB_\K^n$.

\begin{lemma}\label{lemma.intersection-transversality}
	Fix $V \in \DB^n (\K^{n+1})$. For any $U \in \DB^k (\K^{n+1})$ and $W \in \DB^2 (\K^{n+1})$ such that
	\begin{itemize}
		\item $\Span(U) \subset \Span(V)$;
		\item $\Span(W) \not \subset \Span(V)$
	\end{itemize}
	we have
	\begin{equation}
		\Dist{U}{W} \geq \frac{\| U \vee (W \wedge V)\|}{\|U\| \cdot \|W \wedge V \|}
	\end{equation}
\end{lemma}
\begin{proof}
	Let $V=v_1 \vee \cdots \vee v_n$ and assume, as we may, that $(v_1,\ldots,v_n)$ is an orthonormal
	basis of $\Span(V)$. Also, choose $w_2 \in \Span(V)$ and $w_1$ orthogonal to $\Span(V)$ such that
	$W=w_1 \vee w_2$.

	It follows that $W \wedge V$ is colinear to $w_2$; hence
	\begin{equation}
		\frac{\| U \vee (W \wedge V)\|}{\|U\| \cdot \|W \wedge V \|} = \Dist{U}{w_2}
	\end{equation}

	On the other hand,
	\begin{equation}
		\Dist{U}{W} = \Dist{U \vee w_2}{w_1} \times \Dist{U}{w_2} \times d(w_1,w_2)^{-1}
	\end{equation}
	where the first term of the right-hand side is equal to $1$ because $w_1$ was chosen to be orthogonal to $\Span(V)$
	and both $\Span(U)$ and $w_2$ are inside $\Span(V)$; also, we know that $d(w_1,w_2) \leq 1$ (\eqref{eq.angular-metric} and Lemma 
	\ref{lemma.basic-inequality}).

	Hence the Lemma.
\end{proof}

\begin{proposition} \label{prop.second-transversality}
	Fix $V$ as in the Lemma, and let $K$ be some compact subset of $\PB_\K^n \setminus L_V$. For any
	distinct $w_1,w_2 \in K$
	\begin{equation}
		\Leb \{ U \in \PB \DB^{k+1} (\Span(V))\ ;\ \phi_U (w_1,w_2) < r \} \lesssim r^{\delta_\K (n-k-1)}
	\end{equation}
	where as previously $\delta_\K$ is $1$ or $2$ according as $\K$ is $\R$ or $\C$.
\end{proposition}
\begin{proof}
	Follow the line of the proof of Proposition \ref{prop.first-transversality}, replacing
	the $2$-vector $W$ with the (genuine) vector $V \wedge W$.
\end{proof}

\begin{corollary} \label{cor.second-transversality}
	Fix $k$ and $V$ as before. Let $A$ be a Borel subset of $\PB_\K^n \setminus L_V$ of Hausdorff
	dimension $s$. For almost every $k$--dimensional projective subspace $L$ of
	$\PB_\K (\Span(V))$, the transverse dimension of $A$ with respect to the foliation of
	$\PB_\K^n \setminus L_V$ by $(k+1)$--dimensional projective subspaces containing $L$,
	is equal to
	\begin{equation}
		\inf \{ \delta_\K (n-k-1),s\}
	\end{equation}
\end{corollary}

This is very similar to the previous Corollary, but we are now looking at $k$--dimensional
projective subspaces of a fixed projective hyperplane $L_V$, effectively lowering the dimension
of the space of foliations.

More precisely, the space of $k$--dimensional projective subspaces of $\PB_\K^n$ has
dimension $(k+1)(n-k)$ whereas the space of $k$--dimensional projective subspaces of $L_V$
has dimension $(k+1)(n-k-1)$.

In restricting our space of foliations, we did not lose anything dimension-wise. We will see later
that this is not really surprising, by showing how, when $\K=\R$, this Corollary is actually
equivalent to the classical Marstrand--Kaufman--Mattila projection Theorem.

\section{Linear foliations of spheres}\label{s.spheres-foliations}

\subsection{General setup} \label{ss.general-setup}

Let $n \geq 2$. We will deal at the same time with the $(n-1)$-sphere in $\PB_\R^n$ and 
the $(2n-1)$-sphere in $\PB_\C^n$, so let us introduce suitable notations: denote by $\SB, \BB \subset \PB_\K^n$ 
the sets\label{def.S-B}
\begin{equation}
	\left. \begin{array}{lcc}
		\SB^{n-1} & = & \{ [1:x_1:\ldots:x_n] \in \PB_\R^n\ ;\ x_1^2+\cdots+x_n^2 = 1\} \\
		\BB^n     & = & \{ [1:x_1:\ldots:x_n] \in \PB_\R^n\ ;\ x_1^2+\cdots+x_n^2 < 1\}
	\end{array} \right\} \quad \mathrm{if} \quad \K=\R
\end{equation}
\begin{equation}
	\left. \begin{array}{lcc}
		\SB^{2n-1} & = & \{ [1:z_1:\ldots:z_n]\in \PB_\C^n\ ;\ |z_1|^2+\cdots+|z_n|^2 = 1\} \\
		\BB^{2n}   & = & \{ [1:z_1:\ldots:z_n]\in \PB_\C^n\ ;\ |z_1|^2+\cdots+|z_n|^2 < 1\}
	\end{array} \right\} \quad \mathrm{if} \quad \K=\C
\end{equation}

If $L$ is some $k$--dimensional projective subspace of $\PB_\K^n$, let $\mathcal F_L$ be the
foliation of $\SB$ the leaves of which are the intersections of
$\SB$ with $(k+1)$--dimensional projective subspaces of $\PB_\K^n$
containing $L$.

(Remember that we are abusing the language and that what we really mean here is that $\mathcal F_L$
is the foliation of $\SB\setminus (L \cap \SB)$ the leaves of which are the intersections of
$\SB \setminus (L \cap \SB)$ with $(k+1)$--dimensional projective subspaces of $\PB_\K^n$
containing $L$.)

If $\K=\R$, the case $k=0$ is essentially empty and should be removed from consideration.

The leaves of $\mathcal F_L$ are small $k$--spheres if $\K=\R$, respectively small $(2k+1)$-spheres if
$\K=\C$.

Our previous projection results in $\PB_\K^n$ (Corollaries \ref{cor.first-transversality} and \ref{cor.second-transversality}) 
translate without any further work
to interesting projection results in $\SB$. We need only remark that the restriction of $d$
to $\SB$ is equal to the usual angular metric on $\SB$.

\begin{theorem}
	Fix $k$, $0\leq k \leq n-2$, and let $A$ be a Borel subset of $\SB$ of Hausdorff dimension $s$.
	For almost every $k$--dimensional projective subspace $L$ of $\PB_\K^n$, the transverse
	dimension of $A$ with respect to $\mathcal F_L$ is equal to $\inf \{ s, \delta_\K (n-k-1)\}$.
\end{theorem}

We are going to restrict this family of foliations in order to obtain foliations which can be
described purely in terms of spherical geometry.

For $k \geq 1$, let $\mathcal L_\K^k$\label{def.Lk} be the space of $k$--dimensional projective subspaces
$L$ of $\PB_\K^n$ that meet $\BB$. If $\K=\R$, this is the same thing as the  space of small $(k-1)$--spheres
of $\SB^{n-1}$.

\begin{definition}\label{def.kchain}
	Assume $\K=\C$. For $k \geq 0$, a $k$--chain is the intersection of $\SB^{2n-1}$ with
	a $k$--dimensional projective subspace of $\PB_\C^n$ that meets $\BB^{2n}$.
\end{definition}
A $k$--chain is the complex analogue of a small $k$--sphere; it is also a special case of small $(2k-1)$--sphere.

For example, the space of all small $1$-spheres (i.e. all small circles) of $\SB^3$ has dimension $6$, whereas the
space of all $1$-chains of $\SB^3$ has dimension $4$.

\begin{lemma}
	If $k \geq 1$, $\mathcal L_\K^k$ is an open subset of $\PB \DB^{k+1}(\K^{n+1})$, the space of all
	$k$--dimensional projective subspaces of $\PB_\K^n$.
\end{lemma}
In particular, the Lebesgue measure of this space is non-zero. Corollary \ref{cor.first-transversality}
thus implies the following

\begin{theorem}\label{th.sphere-first}
Fix $k$, $1 \leq k \leq n-2$ and let $A$ be a Borel subset of $\SB$ of Hausdorff dimension $s$.
\begin{description}
\item[$\K=\R$ :] For almost every $(k+1)$--tuple $(u_0,\ldots,u_k)$ of points of
$\SB=\SB^{n-1}$, the transverse dimension of $A$, with respect to the foliation of
$\SB^{n-1}$ by small $k$--spheres passing through each of these points, is equal to
\begin{equation}
	\inf\{n-k-1,s\}
\end{equation}
\item[$\K=\C$ :] For almost every $(k+1)$--tuple $(u_0,\ldots,u_k)$ of points of
$\SB=\SB^{2n-1}$, the transverse dimension of $A$, with respect to the foliation of
$\SB^{2n-1}$ by $(k+1)$--chains passing through each of these points, is equal to
\begin{equation}
	\inf\{2(n-k-1),s\}
\end{equation}
\end{description}
\end{theorem}
This follows from Corollary \ref{cor.first-transversality} because the restriction of 
the Lebesgue measure to $\mathcal L_\K^k$ is equivalent to the probability measure obtained 
by picking $k+1$ points at random on $\SB$ and looking at the only $k$--dimensional projective subspace 
passing through each of these points.

In the same fashion, Corollary \ref{cor.second-transversality} implies the following
\begin{theorem}\label{th.sphere-second}
    Fix $k$, $1 \leq k \leq n-2$ and let $L \subset \SB$ be a small $(n-2)$--sphere if $\K=\R$, 
    respectively a $(n-2)$--chain if $\K=\C$. Let $A$ be a Borel subset of $\SB$ of Hausdorff
    dimension $s$.
    \begin{description}
        \item[$\K=\R$ :] For almost every $(k+1)$--tuple $(u_0,\ldots,u_k)$ of points of
        $L$, the transverse dimension of $A$, with respect to the foliation of
        $\SB^{n-1}$ by small $k$--spheres passing through each of these points, is equal to
        \begin{equation}
            \inf\{n-k-1,s\}
        \end{equation}
        \item[$\K=\C$ :]  For almost every $(k+1)$--tuple $(u_0,\ldots,u_k)$ of points of
        $L$, the transverse dimension of $A$, with respect to the foliation of
        $\SB^{2n-1}$ by $(k+1)$--chains passing through each of these points, is equal to
        \begin{equation}
            \inf\{2(n-k-1),s\}
        \end{equation}
    \end{description}
\end{theorem}

When $\K=\C$, the case $k=0$ is missing (because $\SB^{2n-1}$ is negligible for the 
Lebesgue measure on $\PB_\C^n$) and we have to handle it separately.

\subsection{Foliations by $1$-chains} \label{ss.foliations-one-chains}
We now fix $\K=\C$. For every $u \in \SB^{2n-1}\subset\PB^n_\C$, the foliation of $\PB_\C^n$ by projective lines
passing through $u$ induces a foliation of $\SB^{2n-1}$ by 1-chains passing through $u$.

\begin{theorem}\label{th.sphere-third}
Let $A$ be a Borel subset of $\SB^{2n-1}$ of Hausdorff dimension
$s$. For almost every $u \in \SB^{2n-1}$, the transverse  dimension of $A$
with respect to the foliation of $\SB^{2n-1}$ by $1$--chains passing through $u$ is equal to
\begin{equation}
\inf\{s, 2n - 2\}
\end{equation}
\end{theorem}
\begin{proof}
The Hausdorff dimension of a Borel set $A$ is the supremum of the
Hausdorff dimensions of its compact subsets. Using this fact, we can
assume, without loss of generality, that $A$ is a compact subset of $\SB^{2n-1}$.
Let $O$ be an open subset of $\SB^{2n-1}$ that is non-empty and such that the closure
$\overline{O}$ does not meet $A$. We first show that the conclusion holds for almost
every $u \in O$.

Introduce the set $C_\varepsilon(A, O)$ of all projective lines $[W] \in \PB \DB^2(\C^{n+1})$
that meet $A$ and $\overline{O_\varepsilon}$, where $O_\varepsilon$ is the $\varepsilon$--neighbourhood
of $O$ in $\PB^n_\C$, i.e.

Fix $\varepsilon$ small enough that $A$ does not meet $\overline{O_\varepsilon}$. Let $G = \mathbf{PU}(1, n)$ and
 fix as in Lemma \ref{lemma.sphere-third} a compact subset $\mathcal G$ of $G$ such that any 1-chain meeting
$A$ and $\overline{O_\varepsilon}$ is of the form $gL_0$, where $g \in \mathcal G$ and $L_0$ is the 1-chain passing
through, say, $[e_0 + e_1]$ and $[e_0 - e_1]$. (Recall that $(e_0,\ldots,e_n)$ is the canonical basis of $\K^{n+1}$.)

I claim that for any $r > 0$, and any distinct $w_1, w_2 \in A$
\begin{equation}
\Leb_{\SB^{2n-1}} \{u \in O\ ;\ \tau (u, W ) \leq r\} \lesssim r^{2n-2}
\end{equation}
where $W = w_1 \vee w_2$ and the constant implied by the notation $\lesssim$ is uniform
in $r, W$ and $u$.

We can assume that $r$ is small enough (with respect to the $\varepsilon$ fixed
above) that the projective line $[W ]$ has to meet both $A$ and $\overline{O_\varepsilon}$ in order for the left-hand side to be non-zero.

Let $g$ be an element of $\mathcal G$ such that $[W ] = [gW_0]$
where $W_0 = e_0 \vee e_1$. This is possible because $(e_0 + e_1) \vee (e_0 - e_1)$ is a
scalar multiple of $e_0 \vee e_1$.

Now
\begin{multline}
\Leb_{\SB^{2n-1}} \{u \in O\ ;\ \tau (u, gW_0) \leq r\} \lesssim \Leb_{\SB^{2n-1}} \{u \in O\ ;\ \tau (u, W_0) \lesssim r\} \\
 \lesssim  \Leb_{\SB^{2n-1}} \{u \in \SB^{2n-1}\ ;\ \tau (u, W_0) \leq r\} \lesssim r^{2n-2}
\end{multline}
where the first inequality follows from the compacity of $\mathcal G$ (and the 
subsequent fact that the singular values of $g$ belong to a compact subset of
$]0, +\infty[$) and the last inequality is an easy computation.

At this point we can apply Kaufman's argument; this yields that for Lebesgue-almost every $u \in O$, the transverse dimension of $A$ with respect
to the foliation of $\SB^{2n-1}$ by 1-chains passing through $u$ is equal to
$\inf\{s, 2(n - 1)\}$.

Now let $x$ be any point of $\SB^{2n-1}$. For any $\varepsilon > 0$ we can find $\delta > 0$
such that
\begin{equation}
\dim(\K  \setminus B(x, \delta)) \geq s - \varepsilon
\end{equation}
(where $\dim$ is the Hausdorff dimension). Taking into account the previous statement, 
it follows that for Lebesgue-almost every $u \in B(x, \delta/2)$, the transverse dimension of $K$, with respect
to the foliation by 1-chains passing through $u$, is at least equal to $\inf\{s - \varepsilon, 2(n - 1)\}$.

The compacity of $\SB^{2n-1}$ then implies that for Lebesgue-almost every
$u \in \SB^{2n-1}$, the transverse dimension of $K$, with respect to the foliation
by 1-chains passing through $u$, is at least $\inf\{s - \varepsilon, 2(n - 1)\}$.

The Theorem follows by letting $\varepsilon$ go to 0 along a countable sequence.
\end{proof}

\begin{lemma}\label{lemma.sphere-third}
Let $K^-$, $K^+$ be disjoint non-empty compact subsets of $\SB^{2n-1}$
and let $L_0$ be a fixed $1$--chain. There is a compact subset $\mathcal G$ of G such that
any $1$--chain meeting both $K^-$ and $K^+$ is of the form $gL_0$ for some $ g \in \mathcal G$.
\end{lemma}
\begin{proof}
Fix $\xi^-$, $\xi^+$ two distinct elements of $L_0$ and let $KAN$ be an Iwasawa
decomposition of $G$ in which the Cartan subgroup $A$ fixes both $\xi^-$ and
$\xi^+$ (and, consequently, leaves $L_0$ globally invariant). 
The mapping $g \mapsto (g\xi^-, g\xi^+)$ defines, by passing to the quotient, a proper and onto mapping
\begin{equation}
    \omega : G/A \to \{(\eta^-, \eta^+) \in \SB^{2n-1} \times \SB^{2n-1}\ ;\ \eta^- \neq \eta^+ \}
\end{equation}
(Recall that a continuous mapping is proper if the inverse image of any
compact subset is a compact subset.)

Now, $K^- \times K^+$ being compact, so must be its inverse image $\omega^{-1}(K^-\times K^+)$.

To conclude, use the fact that any compact subset of $G/A$ is the image
(through the quotient mapping $G \to G/A$) of a compact subset of $G$.
\end{proof}

\begin{theorem}\label{th.sphere-chain}
     Let $A$ be a Borel subset of $\SB^{2n-1}$ of Hausdorff dimension $s$. Fix a $(n - 1)$--chain $L$. 
     For Lebesgue-almost every $u \in L$, the transverse dimension of A with respect to the foliation 
     of $\SB^{2n-1}$ by 1-chains passing through $u$ is at least
\begin{equation}
    \inf\{s, 2n - 3\}
\end{equation}
\end{theorem}
\begin{proof}
We argue as in the proof of the previous Theorem. Using transitivity of $G=\mathbf{PU}(1, n)$ in the same fashion as before, we can safely assume
that $L$ is the intersection of $\SB^{2n-1}$ with $\PB(\C e_0 \oplus \cdots \oplus \C e_{n-1})$. Let
$V = e_0 \vee \cdots \vee e_{n-1}$.

Now fix a compact subset $K \subset \SB^{2n-1}$ that does not meet $L = L_V$.

For any distinct $w_1, w_2 \in K$, letting $W = w_1 \vee w_2$, we know by Lemma \ref{lemma.intersection-transversality}
that 
\begin{equation}
    \tau (u, W ) \geq  \tau (u, W \vee V )
\end{equation}
for any $u \in \SB^{2n-1}$. It follows that
\begin{equation}
\Leb_L \{ u \in L\ ;\ \tau (u, W ) \leq r\}  \leq \Leb_L \{ u \in L\ ;\ d(u, W \vee V ) \leq r\}
\end{equation}
and $W \vee V$ is just a point of the $(n - 1)$--chain $L$.

The $(n - 1)$--chain $L$ is equal to the $(2n - 3)$--sphere $\SB^{2n-3} \subset \PB_\C^{n-1}$
where we identify $\PB_\C(\Span(V))$ with $\PB_\C^{n-1}$. The angular metric defined on $\PB^{n-1}_\C$
as in formula \eqref{eq.angular-metric} is equal to the restriction of $d$ to $\PB^{n-1}_\C$. Its restriction to 
$L$ is just the spherical metric of $\SB^{2n-3}$. The right-hand side in the previous equation
is thus $\leq r^{2n-3}$. The Theorem follows from this estimate, as above.
\end{proof}

\subsection{A concrete look at these foliations}\label{ss.concrete-look}
We now look at some concrete examples in order to get a better idea of
what is actually going on in the previous Theorems and how our results
are related to Marstrand's classical projection Theorem.

\subsubsection{The real case}
We fix $\K  = \R$. 

\paragraph{In affine coordinates.} Let $k = 0$ and $n \geq  2$. For any $u \in \PB_\R^n$, we
consider the foliation of $\PB_\R^n$ by projective lines passing through $u$
(with $u$ removed). If we pick affine coordinates $\R^n \subset \PB_\R^n$ and send some
projective hyperspace $\PB^{n-1}$ to infinity, we are looking, when $u$ belongs to
$\R^n$,  at the usual radial projection in $\R^n$.

On the other hand, if $u$ belongs to $\PB^{n-1}$, the resulting foliation of $\R^n$
has leaves affine lines parallel to the vector line associated to $u$.

We can now translate the conclusion of Corollary \ref{cor.second-transversality} in affine terms:
for Lebesgue-almost every $u \in \PB^{n-1}$, the dimension of $A$ transverse to
the foliation of $\R^n$ by affine lines parallel to the vector line associated to $u$, is equal to $\inf\{s, n - 1\}$.
 This is exactly the statement of the usual Marstrand projection Theorem along lines.

We see that this statement is equivalent to the following one: for
any affine hyperspace $L$ in $\R^n$, for Lebesgue-almost every $u \in L$, the
dimension of $A$ transverse to the foliation of $\R^n$  by affine lines
passing through $u$ is equal to $\inf\{s, n - 1\}$.

We leave it to the reader to inspect the case when $k \geq 1$ and come to
the conclusion that Corollary \ref{cor.second-transversality} is again essentially equivalent to
Marstrand's classical projection Theorem (which, in this case, is due to Kaufman and
Mattila).

\paragraph{Foliations of the sphere.} If $k = 0$, there is no interesting
 foliation induced because a (real) projective line meets the sphere in at most two points.

Assume $1 \leq k \leq n - 2$. Let us fix already a $(n - 2)$--sphere $L \subset \SB^{n-1}$
and let $A$ be a Borel subset of $\SB^{n-1}$ of Hausdorff dimension $s$. According
to Theorem \ref{th.sphere-second} for almost every $u_0, \ldots, u_k$ in $L$, the transverse dimension
 of $A$ with respect to the foliation of $\SB^{n-1}$ by $k$--spheres passing through
$u_0, \ldots, u_k$ is equal to $\inf\{s, n - 1 - k\}$.

Let us look at this result in the Euclidean space: send $u_0$ at infinity via 
the stereographic projection, in such a way that $L$ is the subspace $\R^{n-2}$ of $\R^{n-1}$.
 For any $u_1, \ldots, u_k \in L$, the foliation of $\R^{n-1}$ we obtain has leaves the affine $k$--spaces containing
$u_1, \ldots, u_k.$ We thus see that Theorem \ref{th.sphere-second} is actually weaker than Corollary \ref{cor.second-transversality} when $\K  = \R$.

\subsubsection{The complex case}
\paragraph{In affine coordinates.} Let us look already at the affine version of
Corollary \ref{cor.second-transversality}. First, fix $k = 0$. We are looking at foliations of $\C^n$ by
complex affine lines parallel to a given complex vector line (associated to
some point $u \in \PB_\C^{n-1}$). We can recast this in real terms: we are looking
at foliations of $\R^{2n}$ with real affine 2-planes parallel to a given real vector 2-plane. 
Now this real vector 2-plane cannot be just any 2-plane: it has
to be the real plane underlying some complex line.

It becomes apparent that we are effectively improving on Marstrand's
projection Theorem. This Theorem deals with the family of every real
vector plane of $\R^{2n}$, and we see that the conclusion still holds when we
restrict to this subspace of foliations, which is is Lebesgue-negligible.

It is perhaps enlightening to compare the dimension of the space of all
foliations of $\R^{2n}$ by parallel affine 2-planes, which is equal to $2(n-1)\times 2 = 4(n-1)$,
to the dimension of the subspace of those special foliations coming from complex lines, 
which is equal to $2(n - 1)$ (this is of course the real dimension of $\PB_\C^{n-1}$).

For an arbitrary $k \geq 0$ (and $k \leq n - 2$), we are looking, in the complex
case, at a family of foliations of $\R^{2n}$ by $2(k + 1)$--dimensional real spaces coming
from $(k + 1)$--dimensional complex spaces; the dimension of this
space of foliations is equal to $2(n - 1 - k)(k + 1)$, whereas the dimension of the space 
of all $2(k + 1)$--dimensional real vector subspaces of $\R^{2n}$ is
equal to $4(n - 1 - k)(k + 1)$.

\paragraph{Foliations of the sphere.} A $k$--chain of $\SB^{2n-1}$ is a special case of
a small $(2k - 1)$--sphere of this sphere. It is only natural to wonder what kind
of object that is. In fact, chains appear naturally in complex hyperbolic
geometry. The complex hyperbolic space of (complex) dimension $n$, $\mathbf{H}^n_\C$,
has $\SB^{2n-1}$ as its boundary at infinity. A totally geodesic submanifold $S$
of $\mathbf{H}^n_\C$ is one of two types:

\begin{itemize}
\item $S$ is isometric to a complex hyperbolic space $\mathbf{H}_\C^k$, $1 \leq k \leq n$;
\item Or $S$ is isometric to a real hyperbolic space $\mathbf{H}_\R^k$, $1 \leq k \leq 2(n - 1)$.
\end{itemize}
In the first case, the boundary of $S$ is a $k$--chain. In the second
case, it is, in Cartan's terminology, a real $k$--sphere. This is quite an unfortunate term; 
a complex chain is a sphere just as much as a real sphere is. 
We will come back to so-called real spheres in a later paper.

\bibliographystyle{plain}
\bibliography{bibli-chains}

\end{document}